\documentclass[sigconf]{acmart}

\renewcommand\footnotetextcopyrightpermission[1]{} 
\usepackage{amsmath, amsthm, amssymb, enumitem}
\usepackage{tikz} 

\newcommand{\FF}{\mathbb{F}}
\newcommand{\ZZ}{\mathbb{Z}}
\newcommand{\QQ}{\mathbb{Q}}
\newcommand{\RR}{\mathbb{R}}
\newcommand{\CC}{\mathbb{C}}
\renewcommand{\AA}{\mathbb{A}}
\newcommand{\PP}{\mathbb{P}}
\DeclareMathOperator{\Proj}{Proj}

\newcommand{\ab}{{(a:b)}}

\newtheorem{theorem}{Theorem} 
\theoremstyle{acmdefinition}
\numberwithin{theorem}{section}
\newtheorem{remark}[theorem]{Remark} 

\begin{document}

\title{Real Space Sextics and their Tritangents}

\author{Avinash Kulkarni}
\affiliation{%
  \institution{Simon Fraser University}
  \streetaddress{8888 University Drive}
  \city{Burnaby}
  \postcode{V5A 1S6}
  \country{Canada}
}
\email{avi_kulkarni@sfu.ca}

\author{Yue Ren}
\affiliation{%
  \institution{Max Planck Institute MIS Leipzig}
  \streetaddress{Inselstrasse 22}
  \city{Leipzig}
  \postcode{04301}
  \country{Germany}
}
\email{yueren@mis.mpg.de}

\author{Mahsa Sayyary Namin}
\affiliation{%
  \institution{Max Planck Institute MIS Leipzig}
  \streetaddress{Inselstrasse 22}
  \city{Leipzig}
  \postcode{04301}
  \country{Germany}
}
\email{mahsa.sayyary@mis.mpg.de}

\author{Bernd Sturmfels}
\affiliation{%
  \institution{Max Planck Institute MIS Leipzig}
  \streetaddress{Inselstrasse 22}
  \city{Leipzig}
  \postcode{04301}
  \country{Germany}
}
\email{bernd@mis.mpg.de}

\renewcommand{\shortauthors}{A. Kulkarni, Y. Ren, M. Sayyary, B. Sturmfels}

\begin{abstract}
  The intersection of a quadric and a cubic surface in 3-space is a canonical curve of genus 4.
  It has 120 complex tritangent planes. We present algorithms
  for computing real tritangents, and we study the associated discriminants.
  We focus on space sextics that arise from del Pezzo surfaces of degree one.
  Their numbers of planes that are tangent at three real points vary widely;
  both 0  and 120 are attained. This solves a problem suggested by Arnold Emch in 1928.
\end{abstract}

\maketitle

\section{Introduction}

We present a computational study of canonical curves of genus~$4$ over the field $\RR$ of
real numbers. Such a curve $C$, provided it is
smooth and non-hyperelliptic, is the complete intersection in $\PP^3$ of a unique surface
$Q$ of degree two and a (non-unique) surface $K$ of degree three.
Conversely, any smooth
complete intersection of a quadric and a cubic in $\PP^3$
is a genus $4$ curve. The degree of $C = Q \cap K$ is six: any plane
in $\PP^3$ meets $C$ in six complex points, counting multiplicity.
We refer to such a curve $C $ as a \emph{space sextic}.

Any space sextic $C$ has at least $120$ complex tritangent planes,
one for each odd theta characteristic of $C$. If the quadric $Q$
is smooth, then these $120$ planes are exactly the tritangents
\cite[Theorem 2.2]{HL18}. However, if $Q$ is singular, then the curve
$C$ has infinitely many tritangents.
We can see this as follows. Any plane
$H$ tangent to $Q$ contains the singular point of $Q$,
and it is tangent to $Q$ at every point in
the line $H \cap Q $. Since the
intersection of $H$ and $C$ is contained in $Q$, the plane $H$ is tangent
to $C$ at every point in $C \cap H$.

In what follows we focus on the case when the
quadric surface $Q$ containing the space sextic $C$ is singular. We adopt the convention that a
\emph{tritangent} of $C$ is one of the $120$ complex planes corresponding to the
odd theta characteristics of $C$. A tritangent is {\em real} if it is defined by a linear
form with real coefficients. A real tritangent is {\em totally real} if it touches
the curve $C$ at three distinct real points.

\begin{figure}[h]
  \captionsetup{singlelinecheck=off}
  \centering
  \begin{tikzpicture}
    \node at (0,0) {\includegraphics[width=0.7\linewidth]{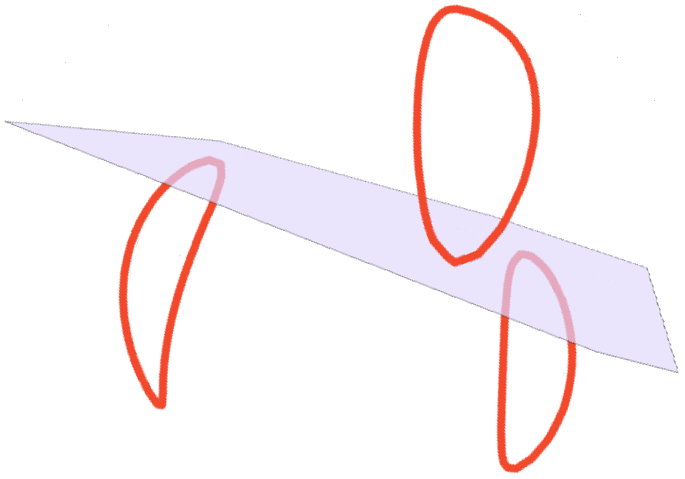}};
    \node[font=\huge] at (1.2,1) {$O_1$};
    \node[font=\large] at (-1.65,-0.4) {$O_2$};
    \node[font=\large] at (1.7,-1.25) {$O_3$};
  \end{tikzpicture}\vspace{-5mm}
  \caption[1]{Totally real tritangent of a curve with three ovals. The plane touches
    $O_1$ on one side and $O_2,O_3$ on the other.}
  \label{fig:quarticsAndBitangents}
\end{figure}
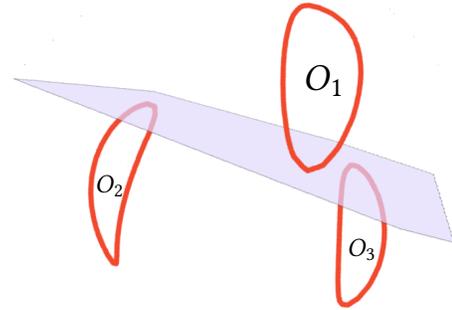

A space sextic $C$ has at most five ovals \cite[\S 3]{HL18},
since the maximum number of ovals is the genus of $C$ plus one.
By \cite[Proposition 3.1]{HL18}, all $120$ tritangents of $C$ are real if and only if
the number of ovals of $C$ attains this upper bound.
A heuristic argument suggests that at least $80 = \binom{5}{3} {\times} 8$
of the $120$ real tritangents are totally real, since eight planes can
touch three ovals as in Figure~\ref{fig:quarticsAndBitangents}.
The analogous fact for genus three curves is true: a plane quartic
with four ovals has $28$ real bitangents, of which at least $24 = \binom{4}{2} {\times} 4$
are totally real. The situation is more complicated in genus $4$,
as seen in Figure~\ref{fig:threeOvals}.

\begin{figure}[h]
  \captionsetup{singlelinecheck=off}
  \begin{center}
    \includegraphics[height=3.5cm]{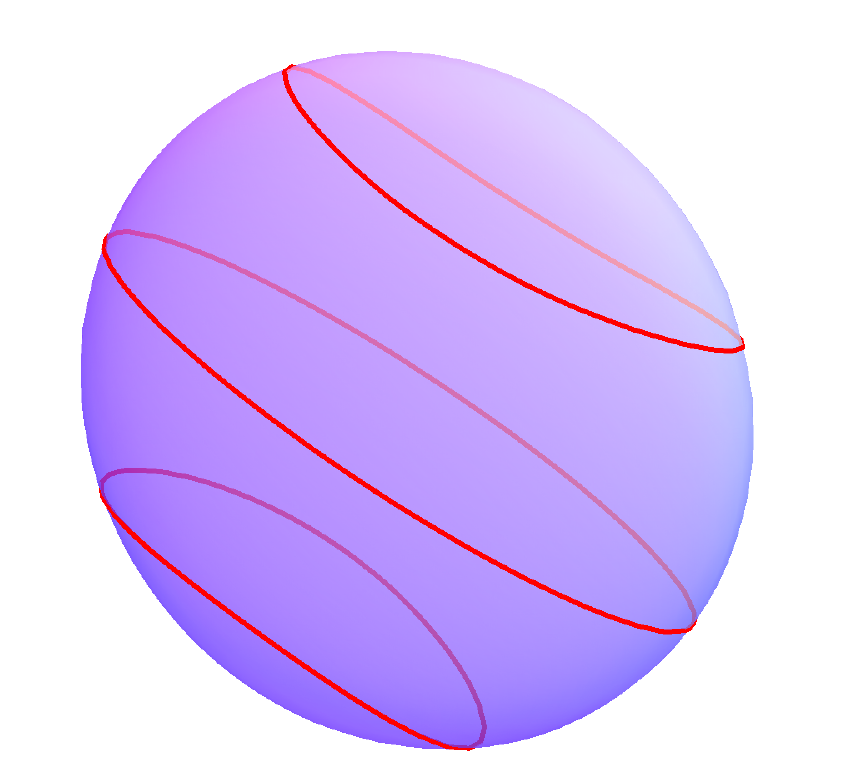}
    \vspace{-3mm}
  \end{center}
  \caption[1]{No tritangent touches all three ovals of this curve.}
  \label{fig:threeOvals}
\end{figure}

In 1928, Emch \cite[\S 49]{Emch28} asked whether there exists a
space sextic with all of its $120$ tritangent planes totally real. He exhibited a curve suspected
to attain the bound $120$. However, ninety years later,
Harris and Len \cite[Theorem 3.2]{HL18} showed that only $108$ of
the tritangents of Emch's curve are totally real.
In \cite[Question 3.3]{HL18} they reiterated the question whether
$120$ totally real tritangents are possible. Our
Example \ref{ex:120} answers that question affirmatively.

\begin{theorem} \label{thm:main}
  The number of totally real tritangents of a space sextic with five ovals
  can be any integer between $84$ and $120$. Each of these numbers is realized
  by an open semialgebraic set of such curves.
\end{theorem}

This article is organized as follows.
In Section~2 we construct space sextics associated with
del Pezzo surfaces of degree one. These curves lie on a singular quadric
$Q$ and are obtained by blowing up eight points in the plane.
This construction has the advantage of producing $120$ rational tritangents
when the points are rational. In Section~2 we also prove Theorem \ref{thm:main}.
In Section~3 we extend this construction to real curves obtained
from complex configurations in $\PP^2$ that are invariant under complex conjugation.
Theorem \ref{thm:table1census} summarizes what we know about these special space sextics.
In Section~4 we turn to arbitrary space sextics, where $Q$ is now generally smooth,
and we show how to compute the $120$ tritangents of
$C = Q \cap K$ directly from the equations defining $Q$ and $K$.
Section~5 offers a study of the discriminants associated with our polynomial system,
and Section~6 sketches some directions for future research. Finally, the scripts
used throughout this article are available at \cite{KNRSscripts}.

\section{Eight Points in the Plane}

We shall employ the classical construction of space sextics from del Pezzo surfaces of degree one.
We describe this construction below and direct the reader to \cite[\S 8]{D12} or \cite[\S 2]{Kulkarni16}
for further details. Any space sextic
$C$ that is obtained from this construction is special:
the quadric $Q$ that contains $C$ is singular.
See also \cite{Kulkarni17}, where these curves $C$ are referred to as
{\em uniquely trigonal genus $4$ curves}.

Fix a configuration $\mathcal{P} = \{P_1,P_2,P_3,P_4,P_5,P_6,P_7,P_8\}$ of
eight points in $\PP^2_\RR$. We may assume that $\mathcal{P}$ is sufficiently
generic to allow for the choices to be made below. Additionally,
genericity of $\mathcal{P}$ ensures that the resulting space sextic $C$
is a smooth curve in $\PP^3$. For practical computations we always choose points $P_i$
whose coordinates are in the field $\QQ$ of rational numbers. This ensures that
each object arising in our computations is defined over $\QQ$.

The space of ternary cubics that vanish on $\mathcal{P}$ is two-dimensional.
We compute a basis $\{u,v\}$ for that space. The space of ternary sextics
that vanish doubly on $\mathcal{P}$ is four-dimensional, and it
contains the three-dimensional subspace spanned by $\{u^2,uv,v^2\}$.
We augment this to a basis by another sextic $w$
that vanishes to order two on $\mathcal{P}$.

The blow-up of $\PP^2$ at the eight points in $\mathcal{P}$ is a
{\em del Pezzo surface} $X_\mathcal{P}$ of degree one.
Our basis $\{u^2, uv , v^2, w\}$ specifies a rational map $\, \PP^2 \dashrightarrow \PP^3 \,$
that is regular outside $\mathcal{P}$ and hence lifts to
$X_\mathcal{P}$. This map is 2-to-1 and its image
is the singular quadric $V(x_0 x_2 - x_1^2)$.
The ramification locus consists of two connected components, the isolated point $(0:0:0:1)$
and the intersection of the quadric $V(x_0 x_2 - x_1^2)$ with a cubic $C$
that is unique modulo $\langle x_0 x_2 - x_1^2 \rangle$.

Following \cite[Example 2.5]{Kulkarni16}, we parametrize the singular quadric
$Q$ as $\{(1:t:t^2:W)\}$. This represents $C$ by
a polynomial in two unknowns $(t,W)$ that has
Newton polygon $\,{\rm conv}\{(0,0), (6,0),(0,3) \}$:
\begin{equation}
  \label{eq:explicitC}
  \begin{aligned}
    \!\!   C : \,\,& t^6 \! + \! c_1 t^5\! + \! c_2 t^4W\! + \!c_3 t^4 \! + \! c_4 t^3 W \! + \! c_5 t^2W^2
    \! + \! c_6 t^3 \! + \! c_7 t^2 W \! +  \\ & \,
    \! c_8 t W^2  \! + \! c_9 W^3  \! + \! c_{10} t^2
    \! + \! c_{11} tW \! + \! c_{12} W^2 \! + \! c_{13} t \! + \! c_{14} W \! + \! c_{15}.
  \end{aligned}
\end{equation}
We derive $120$ tritangents of our curve $C$ in $\PP^3$ from the $240$
exceptional curves on the del Pezzo surface $X_\mathcal{P}$ (cf.~Lemma
\ref{lem:exceptional curves and tritangents}). There is an order two
automorphism $\iota$ of $X_\mathcal{P}$, called the \emph{Bertini involution}. The
image of an exceptional curve $C_1$ under the Bertini involution $\iota$ is another exceptional curve
$C_2 = \iota(C_1)$.
If $\, \varphi\colon X_\mathcal{P} \rightarrow V(x_0 x_2 - x_1^2 ) \,$ is the 2-to-1 covering
branched along $C$, then  $\varphi \circ \iota = \varphi$. In particular, $ \varphi(C_1)
= \varphi(C_2) $. The intersection $C_1 \cap C_2$ consists of three points
on $X_\mathcal{P}$. Their image under $\varphi$
is the triple of points at which the tritangent
corresponding to $\{C_1,C_2\}$ touches $C$.
We can thus decide whether a tritangent is totally real by checking whether
the intersection $C_1 \cap C_2$ in $X_\mathcal{P}$ contains one or three real points.
This intersection can be carried out in $\PP^2$, as we shall explain next.

Recall that $X_\mathcal{P}$ is the blow-up of $\PP^2$ at $\mathcal{P}$. By blowing down,
we may view the eight exceptional fibers of the blow-up as the eight points of $\mathcal{P}$,
and we may view the remaining $112$ exceptional curves of $X_\mathcal{P}$ as (possibly singular) curves in $\PP^2$.
We can determine the images of the exceptional curves in $\PP^2$ from \cite[Table 1]{TVV09}, as well as how they are matched into pairs $\{C_1,C_2\}$
via the Bertini involution:

\begin{enumerate}
\item[8:]
  The exceptional fiber at one point $P_i$ matches the sextic
  vanishing triply at $P_i$ and doubly at the other seven points. The three components
  of the tangent cone of this
  sextic determine the three desired points on the branch curve $C$.
\item[28:] The line through $P_i$ and $P_j$ matches the quintic vanishing at
  all eight points and doubly at the six points in $\mathcal{P} \backslash \{P_i, P_j\}$.
  Their intersection in $\PP^2 \setminus \mathcal P$ consists of three complex points.
  Either one or three of them are real (see Figure~\ref{fig:type28}).
  \begin{figure}[h]
    \centering\vspace{-0.15cm}
    \begin{tikzpicture}[scale=0.2]
      \draw[red] plot[smooth, tension=1] coordinates { (0,0) (4,-2) (9,0) (7,4) (4,2) (9,-2) (14,-1) (15,2) (11,1) (13,-5) (18,-7) (20,-4) (16,-3) (15,-9) (18,-12) (21,-10) (17,-8) (12,-12) (12,-16) (15,-17) (16,-14) (9,-13) (7,-15) (10,-17) (9,-11) (4,-11) (0,-18) };
      \draw (-2,-15) -- (22,-15);
      \fill[red] (7.5,-1.45) circle (8pt);
      \fill[red] (11.25,-2.35) circle (8pt);
      \fill[red] (14.75,-6.6) circle (8pt);
      \fill[red] (15,-8.75) circle (8pt);
      \fill[red] (11.5,-12.9) circle (8pt);
      \fill[red] (10.05,-12.9) circle (8pt);
      \fill (16.75,-15) circle (8pt);
      \fill (7,-15) circle (8pt);
      \draw[fill=white] (1.15,-15) circle (8pt);
      \draw[fill=white] (10.35,-15) circle (8pt);
      \draw[fill=white] (11.45,-15) circle (8pt);
      \node[anchor=north east] at (7,-15) {$P_i$};
      \node[anchor=north west] at (16.75,-15) {$P_j$};
    \end{tikzpicture}\vspace{-0.25cm}
    \caption{$\mathcal{P}$ determines $28$ lines meeting a rational quintic}
    \label{fig:type28}
  \end{figure}
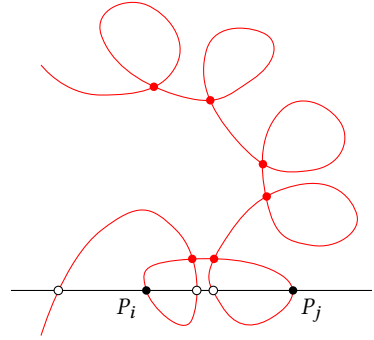
\item[56:] The conic through $P_{i_1}, \ldots, P_{i_5}$ matches the quartic vanishing at $\mathcal{P}$
  and doubly at the three other points. Their
  intersection in $\PP^2 \setminus \mathcal P$ consists of three complex points (see Figure~\ref{fig:type58}).
  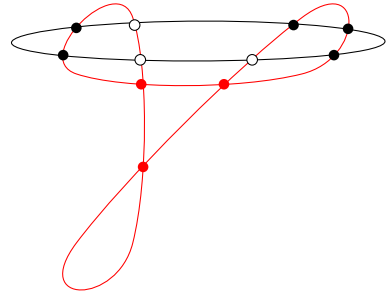
\begin{figure}[h]
    \centering\vspace{-5mm}
    \begin{tikzpicture}[scale=0.25]
      \draw[red] plot[smooth cycle, tension=1] coordinates { (0,0) (12,0) (12,3) (0,-9) (3,-9) (3,3) };
      \draw plot[smooth cycle, tension=1] coordinates { (-1,1) (13.5,1) (13.5,2.5) (0,2.5) };
      \fill[red] (3.45,-0.55) circle (8pt);
      \fill[red] (7.85,-0.55) circle (8pt);
      \fill[red] (3.55,-4.95) circle (8pt);
      \fill (-0.7,1) circle (8pt);
      \fill (0,2.45) circle (8pt);
      \fill (13.7,1) circle (8pt);
      \fill (14.45,2.4) circle (8pt);
      \fill (11.55,2.6) circle (8pt);
      \draw[fill=white] (3.1,2.6) circle (8pt);
      \draw[fill=white] (3.4,0.75) circle (8pt);
      \draw[fill=white] (9.35,0.75) circle (8pt);
    \end{tikzpicture}\vspace{-5mm}
    \caption{$\mathcal{P}$ determines $56$ conics meeting a rational quartic}
    \label{fig:type58}
  \end{figure}
\item[56/2:] For two points $P_i$ and $P_j$, the cubic vanishing doubly at $P_i$, non-vanishing at
  $P_j$, and vanishing singly at $\mathcal{P} \backslash \{P_i,P_j\}$
  matches the cubic vanishing doubly
  at $P_j$, non-vanishing at $P_i$, and vanishing singly at $\mathcal{P} \backslash \{P_i,P_j\}$.
  Their intersection in $\PP^2 \backslash \mathcal{P}$
  consists of three points in $\PP^2$
  (see Figure~\ref{fig:type58/2}).
  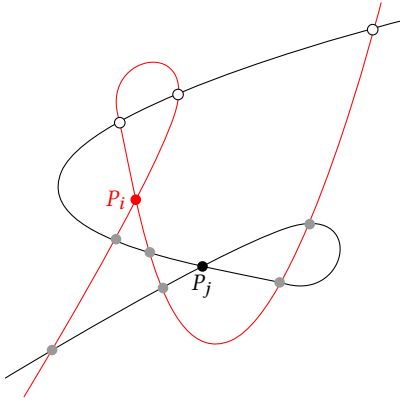
\begin{figure}[h]
    \centering
    \begin{tikzpicture}[scale=0.25]
      \draw[red] plot[smooth, tension=1] coordinates { (-2,-3) (6,12) (3,12) (9,0) (17,18) };
      \draw plot[smooth, tension=1] coordinates { (-3,-2) (12,6) (12,3) (0,9) (18,17) };
      \fill[red] (3.95,7.5) circle (8pt);
      \node[red,anchor=east] at (3.95,7.5) {$P_i$};
      \fill (7.5,3.95) circle (8pt);
      \node[anchor=north] at (7.5,3.95) {$P_j$};
      \fill[black!40] (-0.5,-0.5) circle (8pt);
      \fill[black!40] (5.4,2.8) circle (8pt);
      \fill[black!40] (13.2,6.2) circle (8pt);
      \fill[black!40] (11.6,3.1) circle (8pt);
      \fill[black!40] (4.7,4.7) circle (8pt);
      \fill[black!40] (2.9,5.4) circle (8pt);
      \draw[fill=white] (3.1,11.6) circle (8pt);
      \draw[fill=white] (6.2,13.1) circle (8pt);
      \draw[fill=white] (16.55,16.55) circle (8pt);
    \end{tikzpicture}\vspace{-0.25cm}
    \caption{$\mathcal{P}$ determines $56/2$ pairs of rational cubics
      \label{fig:type58/2}}
  \end{figure}
\end{enumerate}
The following lemma summarizes the reality issues on the del Pezzo surface
$X_\mathcal{P}$ that arises from the constructions
in $\PP^2$ described above.

\begin{lemma} \label{lem:exceptional curves and tritangents}
  Let $\{C_1,C_2\}$ be a pair of exceptional curves of type 8, 28, 56 or 56/2 contained in
  the del Pezzo surface $X_\mathcal{P}$ Then
  $\varphi(C_1\cap C_2)$ spans a tritangent plane of
  the space sextic $C$ in $\PP^3$. That
  tritangent is totally real if and only if
  the intersection $C_1\cap C_2$ is real on $X_\mathcal{P}$.
\end{lemma}
\begin{proof}
  Let $-K$ be the anticanonical divisor class of $X_\mathcal{P}$.
  Then $-K$ and $-2K$ are ample but not very ample.
  The class $-3K$ is very ample, and its linear system embeds $X_\mathcal{P}$
  into $\PP^6$. Consider the sequence of maps
  $\PP^2 \dashrightarrow X_{\mathcal{P}} \rightarrow V(x_0 x_2 - x_1^2) \subset \PP^3$.
  The first map is the blow-up, which is birational. The second map
  is the 2-1 morphism $\varphi$ given by the linear system ${|-2K|}$. The second map takes the
  $240$ exceptional curves in pairs $\{C_1,C_2\}$ onto the
  $120$ hyperplane sections of $V(x_0 x_2 - x_1^2)$ defined by
  the tritangent planes of $C$.

  The pairs are as indicated above, since their classes add up to $-2K$ by
  \cite[Table 1]{TVV09}. Intersection points of the pairs of curves on $X_\mathcal{P}$
  become singular points of the intersection curves on $V(x_0 x_2 -x_1^2)$,
  so the planes are tangent at those points. The tritangent being totally real means
  that these three points have real coordinates.
\end{proof}

In our computations, the del Pezzo surface $X_\mathcal{P}$ is represented
by $(\PP^2,\mathcal{P})$.
For each of the triples of points described above, we can compute their images in
$V(x_0 x_2 - x_1^2) \subset \PP^3$ using Gr\"obner-based elimination.
These triples are the contact points of the corresponding tritangent plane of $C$.
We may choose an affine open subset of $V(x_0 x_2 - x_1^2)$, isomorphic to
$\AA^2$, containing these three points. The intersection of a plane in $\PP^3$
with the singular quadric $Q$ is represented on this open subset by a plane curve with
Newton polygon $\,{\rm conv}\{ (0,0), (2,0), (0,1) \}$. We normalize this as follows:
\begin{equation}
  \label{eq:tritangent}
  \hbox{tritangent planes:} \qquad \qquad
  t^2 \,+ \, e_1 t \,+\,e_2 \,+\, e_3 W .  \qquad
\end{equation}
The upper bound
in Theorem~\ref{thm:main} is attained with Example \ref{ex:120}.

\begin{example} \label{ex:120}
  \rm Consider the following configuration of eight points:
  \begin{align*}
    \mathcal{P} \; = \; \bigl\{\; &(1\! :  \!0\! : \!0),\; (0\! : \!1\! : \!0),\; (0\! : \!0\! : \!1),\; (1\! : \!1\! : \!1),\; (10\! : \!11\! : \!1), \\
                                  & (27\! : \!2\! : \!17),\; (-19\! : \!11\! : \!-12),\; (-15\! : \!-19\! : \!20) \, \bigr\} \;\subset \;\PP^2_\RR.
  \end{align*}
  The resulting space sextic $C$ in
  $V(x_0 x_2 - x_1^2)$ has
  $120$ totally real tritangents.
  We prove this by computing the pairs of special curves in $\PP^2$
  and by computing their triples of intersection points as described above.
  For each of the $112 = 28 + 56 + 56/2$ pairs of curves as above, we found
  that all three intersection points are real. We verified that the remaining
  eight tritangents of $C$ are also totally real by computing the tangent cones
  of the special sextics in item 8.

  \smallskip

  We now convert the curve $C$ to the format in (\ref{eq:explicitC}). From that we
  can recover the pair $(Q,K)$ defining the canonical model of $C$, for the independent verification
  in Example \ref{ex:120b}.
  We start by computing the cubics $u,v$.
  They are minimal generators of the ideal $I:=\bigcap_{i=1}^8 \mathfrak m_{P_i}$, where $\mathfrak m_{P_i}$ denotes the maximal ideal corresponding to the point~$P_i$:
  \begin{align*}
    u \,=\,\,& 7151648400xy^2-434820164119x^2z+354394201544xyz\\[-0.1cm]
             & \quad -38806821565y^2z+692107405715xz^2-580026269975yz^2,\\
    v \,=\,\,& 14303296800x^2y-782195108453x^2z+613370275528xyz-\\[-0.1cm]
             & \quad 49450554755y^2z+1245021817105xz^2-1041049726225yz^2.
  \end{align*}
  Next, we compute the sextic $w$. It is the element of lowest degree in $I^{(2)}\setminus I^2$, where $I^{(2)}$ is the {\em symbolic square} of the ideal $I$. We find
  \begin{align*}
    w \,=\,\,& 175674063641748261863073581969689280x^4yz\\[-0.1cm]
             & \quad +11115515429554564750686439346701440x^3y^2z\\[-0.1cm]
             & \quad -445819563363162103552629662552521920x^2y^3z\\[-0.1cm]
             & \quad +264167833624792096768707005238371200xy^4z\\[-0.1cm]
             & \quad -20036962656454818365487885637968107x^4z^2 \\[-0.1cm]
             & \quad -294913066878605444782558855953184976x^3yz^2\\[-0.1cm]
             & \quad -44062271090476792370117994521819642x^2y^2z^2\\[-0.1cm]
             & \quad +755657199632193956412295956477085200xy^3z^2\\[-0.1cm]
             & \quad -416363969347671237983809688854251675y^4z^2\\[-0.1cm]
             & \quad +32905512814926710254817331888615230x^3z^3\\[-0.1cm]
             & \quad +28993156637165570509985808089578930x^2yz^3\\[-0.1cm]
             & \quad +40808451826702177753226924348677890xy^2z^3\\[-0.1cm]
             & \quad -78682528595564243828185219353313650y^3z^3\\[-0.1cm]
             & \quad -1745283730188673093290045100489475x^2z^4\\[-0.1cm]
             & \quad -5237850029165498581303629066909850xyz^4\\[-0.1cm]
             & \quad -2460237915794525755410066318259875y^2z^4.
  \end{align*}
  The curve $C$ is defined by the generator of the principal ideal
  \[ \left( \left(
        \langle \det J(u,v,w)\rangle + \text{Minors}_{2\times 2}
        \left(\begin{smallmatrix}
            u^2 & uv & v^2 & w \\
            1 & t & t^2 & W
          \end{smallmatrix}\right)\right) : \langle u,v \rangle^2 \right)
    \, \cap \,\QQ[t,W],\]
  where $J(u,v,w)$ is the Jacobian matrix of the map $(x,y,z)\mapsto (u,v,w)$. The determinant of
  $J(u,v,w)$ gives the singular model of the branch curve in $\PP^2$
  and the $2\times 2$ minors determine its image in the singular quadric in $\mathbb P^3$.
  In our case, the generator of the principal ideal is in the form of (\ref{eq:explicitC}),
  and explicitly is given by
  \begin{footnotesize}
    \begin{align*}
      C :\, &\,\,22070179871476654215734436981460373192064947078797748209t^6\\[-0.1cm]
            &+5585831392725719195345163470516310362705889042844010328t^5\\[-0.1cm]
            &+14175569812724447393500233789877848531491265t^4W\\[-0.1cm]
            &-447718078603500717216424896040737869157828321607704039864t^4\\[-0.1cm]
            &-86567655386571901223236593151698362962027440t^3W\\[-0.1cm]
            &+57114529769698357624742306475t^2W^2\\[-0.1cm]
            &+474302309016648096934423520799618219755274954155075926592t^3\\[-0.1cm]
            &+192856342071229007723481356183461213738057680t^2W\\[-0.1cm]
            &-194302706043604453258752959400tW^2 \,
              - \, 26371599148125W^3\\[-0.1cm]
            &+2341397816853864817617847981162945070584483528261510775184t^2\\[-0.1cm]
            &-183528856281941126263893376861009344326329920tW\vphantom{W^2}\\[-0.1cm]
            &+164969244105921949388612135400W^2\\[-0.1cm]
            &-5390258693970772695117811943833419754488807920338145746560t\vphantom{W^2}\\[-0.1cm]
            &+61550499069700173478724063089387654812308400W\vphantom{W^2}\\[-0.1cm]
            &+3193966974265623365398753846860968247266969720956505401600\vphantom{W^2}.
    \end{align*}
  \end{footnotesize}
  We next compute each of the $120$ tritangent planes explicitly, in the format
  (\ref{eq:tritangent}). For instance, the tritangent that arises from the line
  spanned by the points
  $ (10\! : \!11\! : \!1)$ and
  $ (27\! : \!2\! : \!17)$ in $\mathcal{P}$
  is found to be
  \begin{small}
    \begin{align*}
      &  \,345059077005W \,-\,153208173277626716984179949t^2
      \\ &
           \! +277165925195542929496239488t
           \,  - 2613400142391424482367340.
    \end{align*}
  \end{small}
  We now have a list of $120$ such polynomials. Each of these
  intersects the curve $C$ in three complex points
  with multiplicity two in the $(t,W)$-plane.
  All of these complex points are found to be real.
\end{example}

\begin{example} \label{ex:84}
  A similar computation verifies that
  the following configuration of eight points
  gives $84$ totally real tritangents:
  \begin{align*}
    \mathcal P \;=\; \bigl\{\;&(-12: 9: 11)\,,\,\,
                                (7: -5: -7)\,, \,\, (1: 3: 3) \, , \,\, (2: 2: -1) \, ,  \\
                              &(-2: 2: 1) \, , \,\,  (1: 3: 1) \,, \,\, (3: 3: 2) \,, \,\, (8: -8: -7)\; \bigr\} \;\subset \;\PP^2_\RR.
  \end{align*}
\end{example}

\begin{proof}[Proof of Theorem \ref{thm:main}]
  The $120$ tritangent planes arising from the construction above
  correspond to the odd theta characteristics of $C$. They are tritangent to $C$ but
  they do not pass through the singular point $(0:0:0:1)$ of the quadric
  $V(x_0 x_2 - x_1^2)$ in $\PP^3$. Each such tritangent is an isolated
  regular solution to the polynomial equations that define the tritangents of $C$.
  These equations are described explicitly as the tritangent ideal in
  Section~4. We may perturb the equation $x_0 x_2 - x_1^2$ to obtain a new curve $C'$.
  By the Implicit Function Theorem, for each tritangent $H$ of $C$ there is a nearby
  tritangent plane $H'$ of $C'$. Moreover, if the perturbation is sufficiently small
  and the three points of $C \cap H$ are real and distinct, then
  $C' \cap H'$ also consists of three distinct real points.
  Conversely, if two points of $C \cap H$ are distinct and complex conjugate, then
  two points of $C' \cap H'$ will also be distinct and complex conjugate.

  Hence, if our blow-up construction gives $m$ totally real tritangents
  for some $ m \leq 120$ then that same number of real solutions persists throughout some
  open semialgebraic subset in the space $\PP^9_\RR \times \PP^{19}_\RR$ of
  pairs $(Q,K)$ of a real quadric and a real cubic in $\PP^3$.

  Examples \ref{ex:84} and \ref{ex:120} exhibit configurations
  with $m=84$ and $m=120$. Every integer $m$
  between these two values can be realized as well.
  We verified that assertion computationally,
  by constructing a configuration $\mathcal{P}$
  in $\PP^2_\QQ$ for every integer between $84$ and $120$.
\end{proof}

\begin{remark}
	It may be possible to prove by hand that every integer $m$
	between $84$ and $120$ is realizable. The idea is to
	connect the two extreme configurations with a
	general semialgebraic path in $\PP_\RR^9 \times \PP_\RR^{19}$.
	That path crosses the {\em tritangent discriminant} $\Delta_2$
	(cf.~Section 5) transversally.
	At such a crossing point, precisely one of the $120$ configurations marked
	8, 28, 56 or 56/2 fails to have its three intersection points distinct. This means
	that the number of real triples changes by exactly one. So,
	the number of totally real tritangents of the associated space sextic
	changes by exactly one. This is not yet a proof because the path
	might cross the discriminant $\Delta_1$.
\end{remark}

\section{Space Sextics with Fewer Ovals}

In Section 2 we started with eight points in the real projective plane $\PP_\RR^2$.
Here we generalize by taking a configuration $\mathcal{P}$
in the complex projective plane $\PP_\CC^2$ that is invariant under complex conjugation.
This also defines a real curve $C$ in $V(x_0 x_2 - x_1^2 ) \subset \PP_\RR^3$.
To be precise, for $s \in \{1,2,3,4,5\}$, let $\mathcal{P}$ consist of
$2s-2$ real points and $5-s$ complex conjugate pairs.
Such a configuration of eight points defines a real del Pezzo surface $X_\mathcal{P}$.
Additionally, the map $\PP^2 \dashrightarrow \PP^3$
and its branch curve $C$ are defined over $\RR$.
The space sextic $C$ has $s$ ovals and it is not of dividing type
when $s \leq 4$. By, \cite[Proposition 3.1]{HL18}, the number of real tritangents of $C$ equals $2^{s+2}$.
For curves which come from the construction in Section~2, we can derive this number by
examining how complex conjugation acts on
the special curves in $\PP^2_\CC$ we had associated with the point configuration $\mathcal{P}$:

\begin{enumerate}
\item[8:] The exceptional fiber over a point $P_i$ defines a real tritangent if and only if
  the point $P_i$ itself is real.
\item[28:] This tritangent is real if and only if the pair $\{P_i,P_j\}$ is real, i.e.~either
  $P_i$ and $P_j$ are both real, or $P_j$ is the conjugate of $P_i$.
  Among the $28$ pairs, the number of real pairs is thus
  ${\bf 4}=0+4$,  ${\bf 4}=\binom{2}{2}+3$, ${\bf 8} = \binom{4}{2}+2$ and
  ${\bf 16} =  \binom{6}{2}+1$ for $\,s=1,2,3,4$.
\item[56:] This tritangent is real if and only if the triple of singular points in
  the quartic is real. This happens if either the three points are real,
  or there is one real point and a conjugate pair.
  Among the $56$ triples, the number of real triples is thus
  ${\bf 0}$,
  ${\bf 6} = 0+ 2 \cdot 3$,
  ${\bf 12} = \binom{4}{3} + 4 \cdot 2$ ,
  ${\bf 26} = \binom{6}{3} + 6 \cdot 1 $ for $\,s = 1,2,3,4$.
\item[56/2:] In this case, the tritangent is real if and only if the two cubics
  are conjugate, and this happens if and only if
  the pair $\{P_i,P_j\}$ is real. Hence the count is
  ${\bf 4}, {\bf 4}, {\bf 8}, {\bf 16}$, as in the case 28.
\end{enumerate}

For each value of $s \in \{1,2,3,4\}$, if we add up the respective four
numbers then we obtain $2^{s+2}$. For instance, for $s=3$,
the analysis above shows
that $4+8+12+8 = 32$ of the $120 $ tritangents are real.

We wish to know how many of these $2^{s+2}$
real tritangents can be totally real,
as $\mathcal{P}$ ranges over the various types of real configurations.
Our investigations
led to the findings summarized
in Theorem \ref{thm:table1census}.

\begin{theorem} \label{thm:table1census}
  The third row in Table~\ref{tab:fivetypes} lists the ranges of currently known values
  for the number of totally real tritangents
  of real space sextics $C$ that are constructed by blowing up eight points in $\PP^2$:
  \begin{table}[h]
    \vspace{-3mm}
    $$ \begin{matrix}
      s \ \hbox{ovals} && 1 & 2 & 3 & 4 & 5 \\
      \hbox{real} &&
      8 & 16 & 32 & 64  & 120 \\
      \hbox{\centering totally real}
      && [0 ,8] & [1,15] & [10,32] & [35,63]  & [84,120]
    \end{matrix} $$
    \caption{Real and totally real tritangents of a space sextic $C$
      on a singular quadric $Q$, according to number of ovals of $C$.	\label{tab:fivetypes}}\vspace{-5mm}
  \end{table}
\end{theorem}
The following examples exhibit some lower and upper bounds.

\begin{example}[$s=1$] \rm
  Let $\mathcal{P}$ be the following configuration in $\PP^2_\CC$:
  \vspace{-1mm}\begin{align*}
    P_1 =&\; ( i: 1-i: 0), & P_2 &= \overline{P_1},  \\
    P_3 =&\; ( 2-i : -3-i : 3+i ), & P_4 &= \overline{P_3},  \\
    P_5 =&\; ( 2-i : 1-i : -2-i ),  & P_6 &= \overline{P_5}, \\
    P_7 =&\; ( 4i : -i : 4 ), & P_8 &= \overline{P_7}.
  \end{align*}
  The curve $C$ consists of only one oval in $\PP^3_\RR$.
  One checks that {\bf none} of the eight real tritangents of $C$ is
  totally real, i.e.~no plane is tangent to $C$ at three real points.
  On the other hand,
  for the following configuration, {\bf all} eight real tritangents are
  totally real:
  \vspace{-1mm}\begin{align*}
    P_1 & = ( i : 0 : 1 ), &  P_2  & = \overline{P_1}, \\
    P_3  & = (1-3i : -3+2i : 1),  &  P_4 & = \overline{P_3}, \\
    P_5 & = (0 : 2+3i : -3-2i),  &  P_6  & = \overline{P_5}, \\
    P_7 & = (4i : -3+4i: 1+i),  &  P_8  & = \overline{P_7}.
  \end{align*}
\end{example}

\begin{example}[$s=2$] \rm
  We fix the following configuration $\mathcal{P}$
  of two real points and three pairs of complex conjugate points in $\PP^2$:
  \vspace{-1mm}\allowdisplaybreaks\begin{align*}
    P_1 &= (1:   -2i:  2i),  &    P_2 &= \overline{P_1},  \\
    P_3 &= (1: 3+2i:  -3i),  &   P_4 &= \overline{P_3},  \\
    P_5 &= (1+2i:  4+2i: -4+i), &    P_6 &= \overline{P_5}, \\
    P_7 &= (1:   0: -1),  & P_8 &= (0: 4: 1).
  \end{align*}
  The associated curve $C$ has two ovals. Of its $16$ real tritangents,
  {\bf exactly one} is totally real.
  By a random search, we found examples where up to $15$
  of the real tritangents of the curve $C$ are totally real.
  At present, we have not found any $\mathcal{P}$ where the associated
  curve has either $0$ or $16$ totally real tritangents.
\end{example}

Figure~\ref{fig:distribution} shows the empirical distribution we observed
for $s=3$ (left) and $s=4$ (right). The respective ranges are $[10,32]$ and
$[35,63]$.

\begin{figure}[h!]
  \centering\vspace{-3mm}
  \begin{tikzpicture}
    \node[anchor=south east] at (0,0)
    {\includegraphics[height=0.5\linewidth,height=3cm]{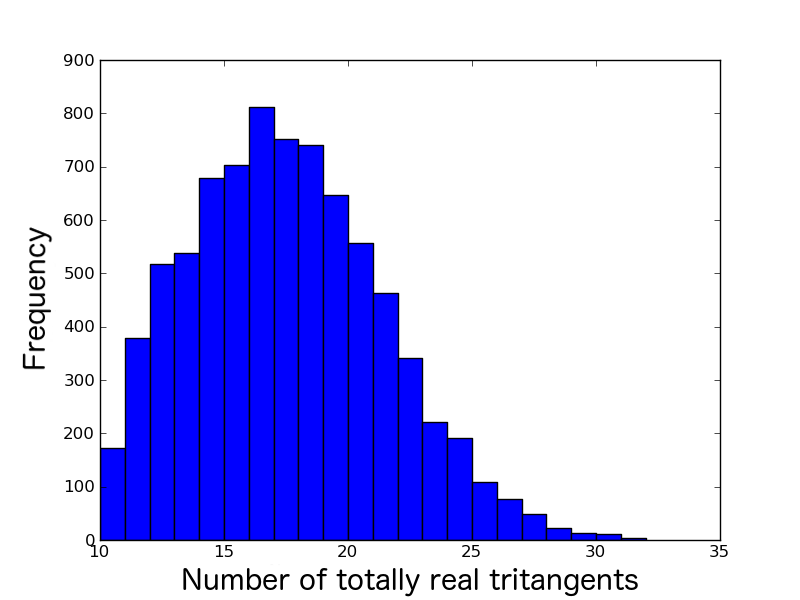}};
    \node[anchor=south west] at (0,0)
    {\includegraphics[width=0.48\linewidth,height=3cm]{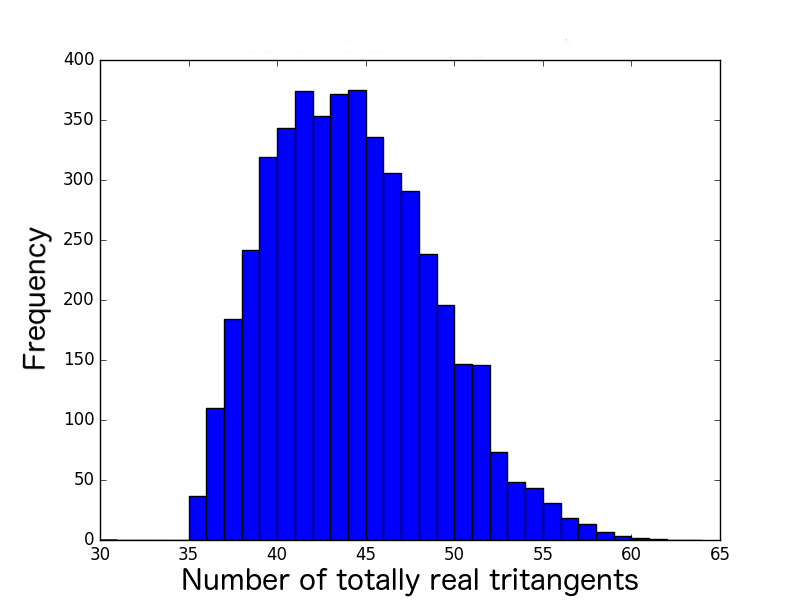}};
  \end{tikzpicture}\vspace{-3mm}
  \caption{Count of totally real tritangents for
    $s=3$ and $s=4$. }\label{fig:distribution}
\end{figure}

\begin{example}[$s=3$] \rm
  The following configuration $\mathcal{P}$ gives a space sextic
  $C$ with three ovals that has
  $32$ totally real tritangents:
  \begin{small}
    \begin{align*}
      P_1 = &\;(- 204813760 -55982740i  : 452442430 + 319792532i  : 1), \; P_2 = \overline{P_1}, \\
      P_3 = &\;( 252002303 -508295920i : 418802957 + 255990940i : 1), \; P_4 = \overline{P_2}, \\
      P_5 = &\;(420794066 : 346448315 : 1), \; P_6 = (64527687 : 183049780 : 1), \\
      P_7 = &\;(410335352 : 364471450 : -1), \; P_8 = (210806629 : 146613813 : -1) .
    \end{align*}
  \end{small}
\end{example}

\section{Solving the Tritangent Equations}

In Sections 2 and 3 we studied
space sextics $C$ lying on a singular quadric surface $Q$.
By perturbing these, we obtained generic space sextics
with many different numbers of totally real tritangents.
However, not all numbers between $0$ and $120$ were attained
by this method. To remedy this, we considered
arbitrary space sextics $C = Q \cap K$,
defined by a random quadric $Q$ and a random cubic $K$.

However, we found the problem of computing the tritangents directly from $(Q,K)$
to be quite challenging. We conjecture that all integers between $0$ and $120$
can be realized by the totally real tritangents of some space sextic.
But, at present, some gaps in Table~\ref{tab:fivetypes} persist.

\smallskip

In what follows we describe our algorithm -- and its implementation --
for computing the $120$ tritangents directly from
the homogeneous polynomials of degree two resp.~three
in $x_0,x_1,x_2,x_3$ that define the quadric $Q$ resp.~the cubic $K$.
We introduce four unknowns $u_0,u_1,u_2,u_3$ that
serve as coordinates on the space $(\PP^3)^\vee$ of planes:
\begin{equation}
  \label{eq:P}
  H \,\,: \,\, u_0 x_0 + u_1 x_1 + u_2 x_2 + u_3 x_3 \,=\, 0.
\end{equation}
For generic real values of the $u_i$, the
intersection $  Q \cap K \cap H  =   C \cap H$ consists of
six distinct complex points in $\PP^3$. We are interested in the special
cases when these six points become three double points.
We seek to find the  {\em tritangent ideal} $\mathcal{I}_C$,
consisting of polynomials in $u_0,u_1,u_2,u_3$
that vanish at those $H$ that are tritangent planes of $C$.

We fix the projective space $\PP^6$ whose points are the binary sextics
$$ f \,\,= \,\, a_0 t_0^6 + a_1 t_0^5 t_1 + a_2 t_0^4 t_1^2
+ a_3 t_0^3 t_1^3 + a_4 t_0^2 t_1^4 + a_5 t_0 t_1^5 + a_6 t_1^6. $$
Inside that $\PP^6$ we consider the threefold of squares of binary cubics:
\begin{equation}
  \label{eq:square}
  f \,\, = \,\, \bigl(b_0 t_0^3 + b_1 t_0^2 t_1 + b_2 t_0 t_1^2 + b_3 t_1^3\bigr)^2.
\end{equation}
The defining prime ideal of that threefold is minimally generated by
$45$ quartics in $a_0,a_1,a_2,a_3,a_4,a_5,a_6$.
This is revealed by the row labeled $\lambda = (2,2,2)$ in \cite[Table 1]{LS16}.
Computing these $45$ quartics is a task of elimination, which
we carried out in a preprocessing step.

Consider now a specific instance $(Q,K)$, defining $C = Q \cap K$.
We then transform the above $45$ quartics in $a_0,\ldots,a_6$ into
higher degree equations in $u_0,\ldots,u_3$.
This is done by projecting $C \cap H$ onto a line.
This gives a univariate polynomial of degree six whose seven coefficients
are polynomials of degree $12$ in $u_0,u_1,u_2,u_3$.
We replace $a_0,\ldots,a_6$ by these polynomials.
Theoretically, it suffices to project onto a single generic line.
Practically, we had more success with multiple (possibly degenerate) projections
onto the coordinate axes, and gathering the resulting systems of $45$ equations each.

To be more precise, fix one of the $12$ ordered pairs $(x_i,x_j)$.
First, solve the equation (\ref{eq:P}) for $x_i$, substitute
into the equations of $Q$ and $K$, and clear denominators.
Next, eliminate $x_j$ from the resulting ternary quadric and cubic.
The result is a binary sextic $f$ in the two unknowns $\{x_0,x_1,x_2,x_3\} \backslash \{x_i,x_j\}$
whose coefficients $a_0,\ldots,a_6$ are expressions of degree $12$ in $u_0,\ldots,u_3$.
We substitute these expressions into the $45$ quartics precomputed above.
This results in $45$ polynomials of degree $48$ in $u_0,\ldots,u_3$
that lie in the tritangent ideal $\mathcal{I}_C$.
Repeating this elimination process for the other $11$ pairs $(x_i,x_j)$,
we obtain additional polynomials in $\mathcal{I}_C$.
Altogether, we have now enough polynomials of degree $48$ to generate
$\mathcal{I}_C$ on any desired affine open subset
in the dual $(\PP^3)^\vee$ of planes in $\PP^3$.
The homogeneous ideal $\mathcal{I}_C$ is radical
and it has $120$ zeros in $(\PP^3)^\vee$.

To compute these zeros, we restrict ourselves to an open chart, say $U=\{u_3\neq 0\}
\simeq \mathbb{C}^3$.
The resulting system (with $u_3=1$) is grossly over-constrained,
with up to $12 \times 45$ equations in the three unknowns $u_0,u_1,u_2$.
We compute a lexicographic Gr\"obner basis, using \texttt{fglm} \cite{FGLM97}, as our ideal is zero-dimensional.
For generic instances $(Q,K)$,  the lexicographic Gr\"obner basis has the shape
\begin{equation}
  \label{eq:shapelemma} \bigl\{\, u_1-p_1(u_3), \,u_2-p_2(u_3),\, p_3(u_3) \,\bigr\},
\end{equation}
where $\deg(p_3) = 120$ and $\deg(p_1) = \deg(p_2) = 119$. For degenerate $(Q,K)$ we proceed with a triangular decomposition.

We implemented this method in \textsc{magma} \cite{magma}.
The Gr\"obner basis computation was very hard to carry out.
It took several days to finish for Example~\ref{ex:Emch}. The
output had coefficients of size ${\sim}10^{680}$.

We applied our implementation to several curves $C$, some from
configurations $\mathcal{P} \subset \PP^2_\QQ$, and some from
general instances $(Q,K)$.

The first case is used as a tool for independent verification,
e.g.~for Example \ref{ex:120}. Here,
$p_3$ decomposes into linear factors over $\QQ$.
Each factor yields a rational tritangent, for which we compute
the three (double) points in $H \cap C$ symbolically. To check whether one or three are real,
we again project onto a line. This yields a univariate rational polynomial of degree 6. We can test whether it is the square of a cubic with positive discriminant. More generally, any non-linear factor with only real roots also allows us to continue our computations symbolically over an algebraic field extension.

In the second case, the univariate polynomial $p_3$ is typically irreducible over $\QQ$, and we
solve (\ref{eq:shapelemma}) numerically.
We compute all real tritangents $H$ and their intersections $H\cap C$.
Based on the resulting numerical data, we decide which $H$ are totally real. Complex zeroes are also counted, to attest that there are indeed $120$ solutions. This certifies that the chosen open chart $U$ was indeed generic.

\begin{example} \label{ex:120b}
  The polynomial $C(t,W)$ in Example~\ref{ex:120} translates
  into a cubic $K(x_0,x_1,x_2,x_3)$ which is unique modulo the quadric
  $Q = x_0 x_2 - x_1^2$. We apply the algorithm above to the instance
  $(Q,K)$ with $U=\{u_3\neq 0\}$. The result verifies that
  all $120$ tritangents are rational and totally real. Interestingly, two of the
  $120$ tritangents have a coordinate that is zero. These two special planes are
  \vspace{-1mm}\begin{align*}
    0 \;&=\; 666727858907928630542805134887161895157 u_0\\
    & \qquad -371406861222752391050720128495402169926 u_1\\
    & \qquad -13148859997292971155483015 u_3
  \end{align*}
  and
  \begin{align*}
    0 \; &=\; 7984878906436628716387308745543788472 u_1 \\
    & \qquad -4446108899575055719305582305633616071 u_2 \\
    & \qquad +10689705055237706452395 u_3.
  \end{align*}
\end{example}

\begin{example}\label{ex:Emch}
  The curve $\,C = Q \cap K\,$ in \cite[\S 3]{HL18} is given by
  \begin{align*}
    Q &\; =\;  x_0^2+x_1^2+x_2^2-25x_3^2,\\
    K &\; =\; \left(x_0+\sqrt 3 x_3\right)\left(x_0-\sqrt 3x_1-3 x_3\right)\left(x_0+\sqrt 3 x_1-3 x_3\right)-2 x_3^3.
  \end{align*}
  It has five ovals, so all tritangents are real. Our computation shows that
  there are only $108$ distinct tritangents. Twelve are solutions
  of multiplicity two in the ideal $I_C$, and none of the tritangents are rational.
  This verifies \cite[Theorem 3.2]{HL18}.
  Figure~\ref{fig:HarrisLenTritangents} shows three tritangents,
  meeting $3$, $2$ and $1$ ovals of the red curve
  respectively.
\end{example}

\begin{figure}[h!]
  \centering
  \begin{tikzpicture}
    \node (a) at (0,0) {\includegraphics[width=0.3\linewidth]{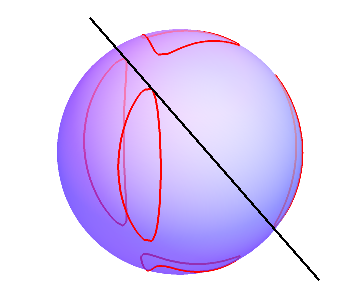}};
    \node[anchor=west] (b) at (a.east) {\includegraphics[width=0.3\linewidth]{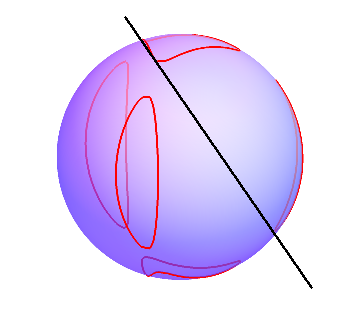}};
    \node[anchor=west] (c) at (b.east) {\includegraphics[width=0.3\linewidth]{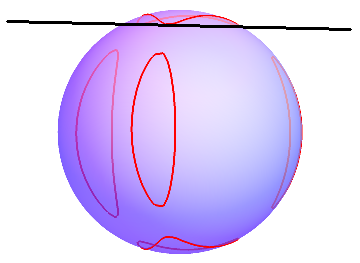}};
  \end{tikzpicture}\vspace{-0.5cm}
  \caption{The curve in Example~\ref{ex:Emch} has $108$ real tritangents}
  \label{fig:HarrisLenTritangents}
\end{figure}

In \cite[Question 3.3]{HL18}, Harris and Len asked whether
this example can be replaced by one with $120$ distinct totally teal
tritangents. Our computations in Examples \ref{ex:120} and \ref{ex:120b}
establish the affirmative answer. However, we do not yet know
whether all integers between $0$ and $120$ are possible
for the number of totally real tritangents.

\section{Discriminants}

In this paper we considered two parameter spaces for space sextics.
First, there is the space $\PP_\RR^9 \times \PP_\RR^{19}$ of pairs
$(Q,K)$ consisting of a real quadric and a real cubic in $\PP^3$.
The regions for which the number of real tritangents remains constant partitions
 $\PP^9_\RR \times \PP_\RR^{19}$ into open strata. This stratification is refined by
 regions for which the number of totally real tritangents remains constant.
We are interested in the discriminantal hypersurfaces
that separate these strata.

Second, there is the space $(\PP_\RR^2)^8$ of configurations $\mathcal{P}$ of
eight labeled points in the plane. This space works for any fixed
value of $s$ in $\{1,2,3,4,5\}$, representing configurations
of $2s-2$ real points and $5-s$ complex conjugate pairs.
For simplicity of exposition we focus on the fully real case $s=5$. In any case,
the number of real tritangents is fixed, and we care about the open strata
in $(\PP_\RR^2)^8$ in which the number of totally real tritangents is constant.
Again, we seek to describe the discriminantal hypersurface, but now in $(\PP^2)^8$.

For $\mathcal{P} \in (\PP^2)^8$, denote the associated space sextic by $C_\mathcal{P}$.
Let $\Sigma$ denote the locus of configurations in $(\PP^2)^8$ which are not in general position. We define the \emph{tritangent discriminant locus } by
\[
  Y \,=\, \overline{\left\{ \,\mathcal{P} \in (\PP^2)^8 \backslash \Sigma:
      \parbox{4.3cm}{ $C_\mathcal{P}$ has a tritangent with contact \\ order at least $4$ at some point} \right\}},
\]
where the over-line denotes the Zariski closure.

\begin{lemma}
  Every irreducible component of $Y$ is a hypersurface.
\end{lemma}

\begin{proof}
	Let $\mathcal{P}_0 \in Y \setminus \Sigma$, and fix local coordinates $\bar p = (p_1, \ldots, p_{16})$ for a neighborhood $\mathcal{U}$ of $\mathcal{P}_0 = \bar p_0$ in $(\PP^2)^8$. The bivariate
	equation (\ref{eq:explicitC}) that represents $C_0$ is
	the specialization at $\bar p_0$ of a general equation
	\begin{equation}\label{eq:ctW}
    \begin{aligned}
      C : \quad & c(t,W) \,=\,  t^6 \! + \! d_1(\bar p) t^5\! + \! d_2(\bar p) t^4W\! + \cdots + \! d_{15}(\bar p),
    \end{aligned}
	\end{equation}
	where the coefficients $d_i(\bar p)$ are rational functions regular at $p_0$. Let $H_0$ be a tritangent plane to $C_0$ with a contact point of order at least $4$. Then $H_0$ is either the tritangent associated to a point in $\mathcal{P}_0$ or associated to one of the patterns in Figure \ref{fig:type28}, \ref{fig:type58} or \ref{fig:type58/2}. Either way, we see that $H_0$ is obtained by specializing an equation of the form
	\begin{equation}\label{eq:htW}
    H\,\colon \,\, \, h(t,W) \,=\, t^2 \,+ \, e_1(\bar p) t \,+\,e_2(\bar p) \,+\, e_3(\bar p) W ,  \qquad
	\end{equation}
	where the coefficients $e_i(\bar p)$ are rational functions regular at $\bar p_0$.

  The resultant of $c(t,W)$ and $h(t,W)$ with respect to $W$ is a polynomial $f(t)$ of degree $6$
  whose coefficients are rational functions in $\bar p$. Note that $H$ is a tritangent plane to $C$,
  so $f = g^2$ as in (\ref{eq:square}).	 The roots of the cubic $g$ correspond to the contact points of $H$ with $C$. In particular, $H_{\bar p}$ has a point of contact with $C_{\bar p}$ of order at least $4$ precisely when the discriminant of $g$ is zero. Since the coefficients of $g$ are rational functions in $\bar p$,
  regular at $\bar p_0$, this means that a neighborhood of
  $\mathcal{P}_0$ in $Y$ has codimension $1$ in $\mathcal{U}$.
  This implies that every irreducible component of $Y$ has codimension $1$.
\end{proof}

The following theorem describes these irreducible components:

\begin{theorem} \label{thm:discriminant}
	The tritangent discriminant locus
	$Y$ is the union~of $120 = 8+28+56+56/2$ irreducible hypersurfaces
	in $(\PP^2)^8$,	one for each point in $\mathcal{P}$ and each pattern
	in Figures \ref{fig:type28}, \ref{fig:type58} and \ref{fig:type58/2}.
	The components of type 8 have total degree $306$, namely $54$
	in the point corresponding to the exceptional curve and $36$ in the other seven points.
	The components of type 28 have total degree $216$, namely $18$
	in each of the two points on the line and $30$ for the six on the quintic.
	The components of type 56 have total degree $162$, namely $18$
	in each of the five points on the conic and $24$ for the three on the quartic.
	The components of type 56/2 have total degree $144$, namely $18$
	in each of the eight points.
\end{theorem}

We prove Theorem \ref{thm:discriminant} computationally. In order to do so, it is convenient to make the following observation. Let $Y = V(f)$ with $f$ a $\ZZ^8$-homogeneous polynomial of
$\ZZ^8$-degree $(d_1, \ldots, d_8)$. We scale $f$ so that its coefficients are relatively prime integers.
For a prime $p$, let $f_p$ denote the reduction of $f$ modulo $p$.
If $p$ is large, then
\[
  Y_p = V(f_p) \subset (\PP_{\FF_p}^2)^8
\]
has the same $\ZZ^8$-degree as $Y$. We can thus calculate $(d_1,\ldots,d_8)$ by
using Gr\"obner bases over a large finite field $\FF_p$.

Let $k= \FF_p$ be the field with $p=10^6+3$ elements and $k^\mathrm{al}$ its algebraic closure. Let $S = \PP_k^1$ and let $R = k[a,b]$ be the coordinate ring of $S$. Let $\PP_R^2 := \Proj R[x,y,z]$ be the projective plane over $R$. If $X$ is some family,
its specialization to $\, \ab \in S \,$ is denoted $X_\ab$.

We use the following configuration of eight points
in $\PP_R^2$:
\begin{align*}
  \mathcal P = \bigl\{\,&(24 : -23: 57),
                          (11:  25: -27),  (-30:  29: 79),  (14: -23: 26),  \\
                        &(43: 92: 61),   (-34: 81: 7),  (88: 29: 69),  (a: b: 0)\, \bigr\} \subset \PP^2_\RR.
\end{align*}
Note $\mathcal{P}$ is in general position for generic $a,b$. Let $\mathcal{U}$ be the open subset
of $S$ parameterizing specializations   in general position.
The following result concerns generic specializations. We omit the proof.

\begin{proposition} \label{prop:discriminant computations}
	There exists a pair of ternary cubics $u,v \in R[x,y,z]$, a ternary sextic
  $w \in R[x,y,z]$, bivariate polynomials $c,h \in R[t,W]$ as in (\ref{eq:ctW}) and (\ref{eq:htW}),
	and an explicitly computable finite set $X \subset S(k^\mathrm{al})$ such that, whenever $\ab \in \mathcal{U} \backslash X$, the following~hold:
	\begin{enumerate}[label={(\alph*)}]
  \item
		The specializations $u_\ab, v_\ab$ span the space of cubics passing through
		all eight points in $\,\mathcal{P}_\ab$.
  \item	 The specializations $u_\ab^2, uv_\ab, v^2_\ab, w_\ab$ span the space of sextics vanishing doubly at each point in $\mathcal{P}_\ab$.
  \item	 The specialization $	\{ c_\ab(t,W) = 0 \}\,$
		is a smooth genus $4$ curve $C_{a:b}$ lying on a singular quadric surface.
  \item	 The specialization $\, \{h_\ab(t,W) = 0\}\,$
		is a tritangent plane to $\mathcal{C}_\ab$ where the coefficient of $W$ is nonzero.
	\item For any $\ab \in X$, the curve $C_{\ab}$ is smooth, genus $4$, and none of the tritangent
    planes have a point of contact order larger than~$4$.
	\end{enumerate}
\end{proposition}

We now derive Theorem \ref{thm:discriminant} from Proposition \ref{prop:discriminant computations}.
The degree $d_8$ of $Y$ in the last point $P_8$ is computed by restricting to the slice
\[
  \{(24  : -23 : 57)\} \times \{(11  :  25 : -27)\} \times \ldots \times \{(88  : 29 : 69)\} \times \PP^2.
\]
This restriction of $Y$ is a curve of degree $d_8$ in $\PP^2$.
We compute this degree as the number of points in the intersection with the line
\[
  S \,\,=\,\, \{ (a : b : c ) \in \PP^2 : c=0 \}.
\]
The same argument works also for each irreducible component of $Y$.
These components correspond to the various tritangent patterns,
marked 8, 28, 56 and 56/2. We perform this computation for
each pattern over $\FF_p$, and we obtain the numbers stated
in Theorem \ref{thm:discriminant}.

\smallskip

We now turn to the canonical representation of arbitrary space sextics
$C = Q \cap K$, namely by pairs $(Q,K)$ in $\PP^9 \times \PP^{19}$.
We shall identify three irreducible hypersurfaces in $\PP^9 \times \PP^{19}$
that serve as discriminants for different scenarios of how $C$ can degenerate.
For each hypersurface, we shall determine its {\em bidegree} $(\alpha,\beta)$.
Here $\alpha$ is the degree of its defining polynomial in the coefficients of $Q$,
and $\beta$ is the degree of its defining polynomial in the coefficients of $K$.

First, there is the classical discriminant $\Delta_1$, which parametrizes
all pairs $(Q,K)$ such the curve $C = Q \cap K$ is singular.
This is an irreducible hypersurface in $\PP^9 \times \PP^{19}$, revisited
recently in \cite{BN17}. The general points of $\Delta_1$
are irreducible curves $C$ of arithmetic genus $4$ that have one simple node,
so the geometric genus of $C$ is~$3$.
The discriminant $\Delta_1$ specifies the wall to be crossed when the number of real
tritangents changes as $(Q,K)$ moves throughout $\,\PP^9_\RR \times \PP^{19}_\RR$.

Second, there is the wall to be crossed when the number of totally real tritangents changes.
The discriminant $\Delta_2$ comprises space sextics with a tritangent $H$
that is degenerate, in the sense that $H$ is tangent at one point and doubly
tangent at another point of $C$. For real pairs $(C,H)$, such a point of double tangency
deforms into two contact points of a tritangent $H_\epsilon$ at a nearby curve $C_\epsilon$,
and this pair is either real or complex conjugate.
On the hypersurface in $\PP^9 \times \PP^{19}$
where $Q$ is singular, the locus $\Delta_2$ is the image of
the discriminant with $120$ components in Theorem \ref{prop:discriminant computations}
under the map that takes
a configuration $\mathcal{P} \in (\PP^2)^8$ to its
associated curve $C_\mathcal{P}$.

Our third discriminant $\Delta_3$ parametrizes
pairs $(Q,K)$ such that the curve $C = Q \cap K$ has
two distinct tritangents that share a common contact point on $C$.
In other words, the curve $C$ has a point whose tangent line
is contained in two tritangent planes. The discriminant $\Delta_3$
furnishes an embedded realization of the {\em common contact locus} that was studied
in the dissertation of Emre Sert\"oz \cite[\S 2.4]{EmreDiss}.

The following theorem was found with the help of
Gavril Farkas and Emre Sert\"oz. The numbers are
derived from results in \cite{FaVe14, EmreDiss}.

\begin{theorem} \label{thm:bidegrees}
  The discriminantal loci $\Delta_1$, $\Delta_2$ and $\Delta_3$
  are irreducible and reduced hypersurfaces in $\,\PP^9 \times \PP^{19}$. Their bidegrees are
  $$  \begin{matrix}
    {\rm bidegree}(\Delta_1) & =  &  (33,34), \\
    {\rm bidegree}(\Delta_2) & = & (744, 592), \\
    {\rm bidegree}(\Delta_3) & = & (8862, 5236).
  \end{matrix}
  $$
\end{theorem}

\begin{proof}
  Consider the discriminant $\Delta_1$
  for curves in $\PP^3$ that are intersections of two surfaces of degree $d$ and $e$.
  It has bidegree
  $$
  \bigl(  \,e(3d^2{+}2de{+}e^2{-}8d{-}4e{+}6)\,, \,
  d(3e^2{+}2de{+}d^2{-}8e{-}4d{+}6) \,\bigr).
  $$
  This can be found in many sources, including
  \cite[Proposition 3]{BN17}.
  For $d=2$ and $e=3$ we obtain $\,{\rm bidegree}(\Delta_1) = (33,34)$, as desired.

  To determine the other two bidegrees, we employ known facts from the
  enumerative geometry of $\overline{\mathcal{M}}_4$,
  the moduli space of stable curves of genus~$4$.
  The Picard group ${\rm Pic}(\overline{\mathcal{M}}_4)$
  is generated by four classes $\lambda,\delta_0,\delta_1,\delta_2$.
  Here $\lambda$ is the {\em Hodge class}, and the $\delta_i$ are classes
  of irreducible divisors in the boundary $\overline{\mathcal{M}}_4 \backslash \mathcal{M}_4$.
  They represent:
  \begin{itemize}
  \item[$\delta_0$:] a genus $3$ curve that self-intersects at one point;
  \item[$\delta_1$:] a genus $1$ curve intersects a genus $3$ curve at one point;
  \item[$\delta_2$:] two genus $2$ curves intersect at one point.
  \end{itemize}
  Our discriminants $\Delta_i$ are the inverse images
  of known irreducible divisors in the moduli space under the rational map
  $\PP^9 \times \PP^{19} \dashrightarrow \overline{\mathcal{M}}_4$.

  First, $\Delta_2$ is the pull-back of the divisor $D_4 \subset \overline{\mathcal{M}}_4$ of
  curves with degenerate odd spin structures. It follows from \cite[Theorem 0.5]{FaVe14}~that
  \begin{equation}
    \label{eq:farkasverra}
    [D_4] \,\,=\,\,1440 \lambda-152 \delta_0 - \alpha \delta_1 - \beta  \delta_2  \quad
    \hbox{for some $\alpha,\beta \in \mathbb{N}$}.
  \end{equation}
  For any curve $\gamma \subset \overline{\mathcal{M}}_4$, the
  sum $\sum_{i=0}^2 \gamma \cdot \delta_i$ counts points on $\gamma$
  whose associated curve is singular.
  Write $h$ resp.~$v$ for the curve $\gamma$
  that represents $\,{\it line} \times {\it point}\,$ resp.~$\,{\it point} \times {\it line}\,$ in
  $\PP^9 \times \PP^{19}$. We saw
  $$(h \cdot \delta_0, v \cdot \delta_0) \,=\, {\rm bidegree}(\Delta_1) \,=\, (33,34). $$
  Moreover, it can be shown that
  $$ h \cdot \lambda = v \cdot \lambda = 4 \qquad \hbox{and} \qquad
  h \cdot \delta_i = v \cdot \delta_i = 0 \quad \hbox{for $\,i=1,2$. } $$
  This implies the assertion about the bidegree of our discriminant:
  $$
  {\rm bidegree}(\Delta_2) =
  (h \cdot [D_4], v\cdot [D_4]) =
  (1440 \cdot 4 - 152 \cdot 33,
  1440 \cdot 4 - 152 \cdot 34).
  $$

  Similarly, $\Delta_3$ is the pull-back of the {\em common contact divisor}
  $Q_4 \subset \overline{\mathcal{M}}_4$ studied by
  Sert\"oz. It follows from \cite[Theorem II.2.43]{EmreDiss} that
  \begin{equation}
    \label{eq:emre}
    [Q_4] \,\,=\,\, 32130 \lambda- 3626 \delta_0 - \alpha \delta_1 - \beta  \delta_2  \quad
    \hbox{for some $\alpha,\beta \in \mathbb{N}$}.
  \end{equation}
  Replacing (\ref{eq:farkasverra}) with (\ref{eq:emre}) in our argument,
  we find that
  $ {\rm bidegree}(\Delta_3) $~is
  $$ (h \cdot [Q_4], v\cdot [Q_4]) \,=\,
  (32130 \cdot 4 - 3626 \cdot 33,\,
  32130 \cdot 4 - 3626 \cdot 34). $$
  This completes our derivation of the bidegrees in Theorem~\ref{thm:bidegrees}.

  The irreducibility of the loci $\Delta_i$ is shown by a standard
  double-projection argument. One marks the relevant
  special point(s) on~$C$. Then $\Delta_i$
  becomes a family of linear spaces of fixed dimension.
\end{proof}

\section{What Next?}

In this paper, we initiated the computational study of
totally real tritangents of space sextics in $\PP^3$.
These objects are important in algebraic geometry because they
represent odd theta characteristics of canonical curves of genus $4$.
We developed systematic tools for constructing curves all of whose
tritangents are defined over algebraic extensions of $\QQ$, and we used this
to answer the longstanding question whether the
upper bound of $120$ totally real tritangent planes can be attained.
We argued that computing the tritangents directly from
the representation $C = Q \cap K$ is hard,
and we characterized the discriminants for these polynomial systems.

\smallskip

This article leads to many natural directions to be explored next.
We propose the following eleven specific problems for further study.

\begin{enumerate}[leftmargin=*]
\setlength\itemsep{1mm}
\item Decide whether every integer between $0$ and $120$ is realizable.
\item Determine the correct upper and lower bounds in Table \ref{tab:fivetypes}.
  In particular, is $84$ the lower bound for curves with five ovals?
\item A smooth quadric $Q$ is either an ellipsoid or a hyperboloid.
  Degtyarev and Zvonilov \cite{DZ} characterized the topological types
  of real space sextics on these surfaces.
  What are the possible numbers of totally real tritangents for their types?
\item What does \cite{DZ} tell us about space sextics on
  a singular quadric $Q$? Which types arise on $Q$, how do they deform to
  those on a hyperboloid, and what does this imply for tritangents?
\item Given a space sextic $C$ whose quadric $Q$ is singular, how to best compute
  a configuration $\mathcal{P} \in (\PP^2)^8$ such that $C = C_{\mathcal{P}}$?
  Our idea is to design an algorithm based on
  the constructions described in \cite[Proposition 4.8 and Remark 4.12]{Kulkarni17}.
\item Lehavi \cite{Lehavi} shows that a general space sextic $C$ can be reconstructed from its $120$ tritangents.
  How to do this in practice?
\item Let $C_\mathcal{P}$ be the space sextic of a configuration $\mathcal{P} \in (\PP^2)^8$.
  How to see the ovals of $C_\mathcal{P}$ in $\PP^2$? For each tritangent as in
  Figure \ref{fig:type28}, \ref{fig:type58} or \ref{fig:type58/2}, how to see
  the number of ovals it touches?
\item Design a custom-tailored {\em homotopy algorithm} for numerically computing the
  $120$ tritangents from the pair $(Q,K)$.
\item The {\em tropical limit} of a space sextic has $15$ classes of tritangents,
  each of size eight \cite[Theorem 5.2]{HL18}. This is realized classically
  by a {\em $K_{3,3}$-curve}, obtained by taking $K$
  as three planes tangent to a smooth quadric $Q$.
  How many totally real tritangents are possible in the vicinity
  of $(Q,K) $ in $ \PP^9_\RR \times \PP^{19}_\RR$~?
\item The $28$ bitangents of a plane quartic are the off-diagonal
  entries of a symmetric $8 \times 8$-matrix, known as the
  {\em bitangent matrix} \cite{Manni}. How to generalize this to
  genus $4$? Is there such a canonical matrix (or tensor)
  for the $120$ tritangents?
\item What is maximal number of $2$-dimensional faces in
  the convex hull of a space sextic in $\RR^3$?
  There are at most $120$ such facets.
In addition, there are infinitely many edges. These form a
ruled surface of degree $54$, by \cite[Theorem 2.1] {Ranestad}.
\end{enumerate}
\vfill
Between the initial and the final version of this paper, much progress was made on Question (2) in \cite{Kummer18,HKSS18}, and Question (11) was answered in \cite{Kummer18}: there are at most $8$ facets.

\bibliographystyle{ACM-Reference-Format}
\bibliography{tritangents}


\begin{thebibliography}{19}


\ifx \showCODEN    \undefined \def \showCODEN     #1{\unskip}     \fi
\ifx \showDOI      \undefined \def \showDOI       #1{#1}\fi
\ifx \showISBNx    \undefined \def \showISBNx     #1{\unskip}     \fi
\ifx \showISBNxiii \undefined \def \showISBNxiii  #1{\unskip}     \fi
\ifx \showISSN     \undefined \def \showISSN      #1{\unskip}     \fi
\ifx \showLCCN     \undefined \def \showLCCN      #1{\unskip}     \fi
\ifx \shownote     \undefined \def \shownote      #1{#1}          \fi
\ifx \showarticletitle \undefined \def \showarticletitle #1{#1}   \fi
\ifx \showURL      \undefined \def \showURL       {\relax}        \fi
\providecommand\bibfield[2]{#2}
\providecommand\bibinfo[2]{#2}
\providecommand\natexlab[1]{#1}
\providecommand\showeprint[2][]{arXiv:#2}

\bibitem[\protect\citeauthoryear{Bosma, Cannon, and Playoust}{Bosma
  et~al\mbox{.}}{1997}]%
        {magma}
\bibfield{author}{\bibinfo{person}{Wieb Bosma}, \bibinfo{person}{John Cannon},
  {and} \bibinfo{person}{Catherine Playoust}.} \bibinfo{year}{1997}\natexlab{}.
\newblock \showarticletitle{The {M}agma algebra system. {I}. {T}he user
  language}.
\newblock \bibinfo{journal}{{\em J. Symbolic Comput.\/}} \bibinfo{volume}{24},
  \bibinfo{number}{3-4} (\bibinfo{year}{1997}), \bibinfo{pages}{235--265}.
\newblock
\newblock
\shownote{Computational algebra and number theory (London, 1993).}


\bibitem[\protect\citeauthoryear{Bus\'e and Nonkan\'e}{Bus\'e and
  Nonkan\'e}{2015}]%
        {BN17}
\bibfield{author}{\bibinfo{person}{Laurent Bus\'e} {and}
  \bibinfo{person}{Ibrahim Nonkan\'e}.} \bibinfo{year}{2015}\natexlab{}.
\newblock \showarticletitle{Discriminants of complete intersection space
  curves}.
\newblock In \bibinfo{booktitle}{{\em I{SSAC}'17---{P}roceedings of the 2017
  {ACM} {I}nternational {S}ymposium on {S}ymbolic and {A}lgebraic
  {C}omputation}}. \bibinfo{publisher}{ACM, New York}.
\newblock
\showeprint{1702.01694}


\bibitem[\protect\citeauthoryear{Dalla~Piazza, Fiorentino, and
  Salvati~Manni}{Dalla~Piazza et~al\mbox{.}}{2017}]%
        {Manni}
\bibfield{author}{\bibinfo{person}{Francesco Dalla~Piazza},
  \bibinfo{person}{Alessio Fiorentino}, {and} \bibinfo{person}{Riccardo
  Salvati~Manni}.} \bibinfo{year}{2017}\natexlab{}.
\newblock \showarticletitle{Plane quartics: the universal matrix of
  bitangents}.
\newblock \bibinfo{journal}{{\em Israel J. Math.\/}} \bibinfo{volume}{217},
  \bibinfo{number}{1} (\bibinfo{year}{2017}), \bibinfo{pages}{111--138}.
\newblock


\bibitem[\protect\citeauthoryear{Degtyarev and Zvonilov}{Degtyarev and
  Zvonilov}{1999}]%
        {DZ}
\bibfield{author}{\bibinfo{person}{A.I. Degtyarev} {and} \bibinfo{person}{V.I.
  Zvonilov}.} \bibinfo{year}{1999}\natexlab{}.
\newblock \showarticletitle{Rigid isotopy classification of real algebraic
  curves of bidegree (3,3) on quadrics}.
\newblock \bibinfo{journal}{{\em Mathematical Notes\/}}  \bibinfo{volume}{66}
  (\bibinfo{year}{1999}), \bibinfo{pages}{670--674}.
\newblock


\bibitem[\protect\citeauthoryear{Dolgachev}{Dolgachev}{2012}]%
        {D12}
\bibfield{author}{\bibinfo{person}{Igor~V. Dolgachev}.}
  \bibinfo{year}{2012}\natexlab{}.
\newblock \bibinfo{booktitle}{{\em Classical Algebraic Geometry: A Modern
  View}}.
\newblock \bibinfo{publisher}{Cambridge University Press}. xii+639 pages.
\newblock
\showISBNx{978-1-107-01765-8}


\bibitem[\protect\citeauthoryear{Emch}{Emch}{1928}]%
        {Emch28}
\bibfield{author}{\bibinfo{person}{Arnold Emch}.}
  \bibinfo{year}{1928}\natexlab{}.
\newblock \showarticletitle{Mathematical models}.
\newblock \bibinfo{journal}{{\em Univ. of Illinois Bull.\/}}
  \bibinfo{volume}{XXV}, \bibinfo{number}{43} (\bibinfo{year}{1928}),
  \bibinfo{pages}{5--38}.
\newblock


\bibitem[\protect\citeauthoryear{Farkas and Verra}{Farkas and Verra}{2014}]%
        {FaVe14}
\bibfield{author}{\bibinfo{person}{Gavril Farkas} {and}
  \bibinfo{person}{Alessandro Verra}.} \bibinfo{year}{2014}\natexlab{}.
\newblock \showarticletitle{The geometry of the moduli space of odd spin
  curves}.
\newblock \bibinfo{journal}{{\em Ann. of Math. (2)\/}} \bibinfo{volume}{180},
  \bibinfo{number}{3} (\bibinfo{year}{2014}), \bibinfo{pages}{927--970}.
\newblock
\showISSN{0003-486X}


\bibitem[\protect\citeauthoryear{Faug\`ere, Gianni, Lazard, and Mora}{Faug\`ere
  et~al\mbox{.}}{1993}]%
        {FGLM97}
\bibfield{author}{\bibinfo{person}{J.~C. Faug\`ere}, \bibinfo{person}{P.
  Gianni}, \bibinfo{person}{D. Lazard}, {and} \bibinfo{person}{T. Mora}.}
  \bibinfo{year}{1993}\natexlab{}.
\newblock \showarticletitle{Efficient computation of zero-dimensional
  {G}r\"obner bases by change of ordering}.
\newblock \bibinfo{journal}{{\em J. Symbolic Comput.\/}} \bibinfo{volume}{16},
  \bibinfo{number}{4} (\bibinfo{year}{1993}), \bibinfo{pages}{329--344}.
\newblock


\bibitem[\protect\citeauthoryear{Harris and Len}{Harris and Len}{2018}]%
        {HL18}
\bibfield{author}{\bibinfo{person}{Corey Harris} {and} \bibinfo{person}{Yoav
  Len}.} \bibinfo{year}{2018}\natexlab{}.
\newblock \showarticletitle{Tritangent planes to space sextics: the algebraic
  and tropical stories}. In \bibinfo{booktitle}{{\em Combinatorial Algebraic
  Geometry}}, \bibfield{editor}{\bibinfo{person}{G.G. Smith} {and}
  \bibinfo{person}{B.~Sturmfels}} (Eds.). \bibinfo{publisher}{Fields Inst. Res.
  Math. Sci.}, \bibinfo{pages}{47--63}.
\newblock


\bibitem[\protect\citeauthoryear{Hauenstein, Kulkarni, {Sert\"oz}, and
  Sherman}{Hauenstein et~al\mbox{.}}{2018}]%
        {HKSS18}
\bibfield{author}{\bibinfo{person}{Jonathan Hauenstein},
  \bibinfo{person}{Avinash Kulkarni}, \bibinfo{person}{Emre~Can {Sert\"oz}},
  {and} \bibinfo{person}{Samantha Sherman}.} \bibinfo{year}{2018}\natexlab{}.
\newblock \bibinfo{title}{Certifying reality of projections}.
\newblock   (\bibinfo{year}{2018}).
\newblock
\showeprint{1804.02707}


\bibitem[\protect\citeauthoryear{Kulkarni}{Kulkarni}{2016}]%
        {Kulkarni16}
\bibfield{author}{\bibinfo{person}{Avinash Kulkarni}.}
  \bibinfo{year}{2016}\natexlab{}.
\newblock \bibinfo{title}{An explicit family of cubic number fields with large
  $2$-rank of the class group}.
\newblock   (\bibinfo{year}{2016}).
\newblock
\showeprint{1610.07668}


\bibitem[\protect\citeauthoryear{Kulkarni}{Kulkarni}{2017}]%
        {Kulkarni17}
\bibfield{author}{\bibinfo{person}{Avinash Kulkarni}.}
  \bibinfo{year}{2017}\natexlab{}.
\newblock \bibinfo{title}{An arithmetic invariant theory of curves from $E_8$}.
\newblock   (\bibinfo{year}{2017}).
\newblock
\showeprint{1711.08843}


\bibitem[\protect\citeauthoryear{Kulkarni, Sayyary, Ren, and
  Sturmfels}{Kulkarni et~al\mbox{.}}{2017}]%
        {KNRSscripts}
\bibfield{author}{\bibinfo{person}{Avinash Kulkarni}, \bibinfo{person}{Mahsa
  Sayyary}, \bibinfo{person}{Yue Ren}, {and} \bibinfo{person}{Bernd
  Sturmfels}.} \bibinfo{year}{2017}\natexlab{}.
\newblock \bibinfo{title}{Data and scripts for this article}.
\newblock \bibinfo{howpublished}{Available at: \url{software.mis.mpg.de}}.
  (\bibinfo{year}{2017}).
\newblock


\bibitem[\protect\citeauthoryear{Kummer}{Kummer}{2018}]%
        {Kummer18}
\bibfield{author}{\bibinfo{person}{Mario Kummer}.}
  \bibinfo{year}{2018}\natexlab{}.
\newblock \bibinfo{title}{Totally real theta characteristics}.
\newblock   (\bibinfo{year}{2018}).
\newblock
\showeprint{1802.05297}


\bibitem[\protect\citeauthoryear{Lee and Sturmfels}{Lee and Sturmfels}{2016}]%
        {LS16}
\bibfield{author}{\bibinfo{person}{Hwangrae Lee} {and} \bibinfo{person}{Bernd
  Sturmfels}.} \bibinfo{year}{2016}\natexlab{}.
\newblock \showarticletitle{Duality of multiple root loci}.
\newblock \bibinfo{journal}{{\em J. Algebra\/}}  \bibinfo{volume}{446}
  (\bibinfo{year}{2016}), \bibinfo{pages}{499--526}.
\newblock
\showISSN{0021-8693}


\bibitem[\protect\citeauthoryear{Lehavi}{Lehavi}{2015}]%
        {Lehavi}
\bibfield{author}{\bibinfo{person}{David Lehavi}.}
  \bibinfo{year}{2015}\natexlab{}.
\newblock \showarticletitle{Effective reconstruction of generic genus 4 curves
  from their theta hyperplanes}.
\newblock \bibinfo{journal}{{\em Int. Math. Res. Not. IMRN\/}}
  \bibinfo{volume}{19} (\bibinfo{year}{2015}), \bibinfo{pages}{9472--9485}.
\newblock


\bibitem[\protect\citeauthoryear{Ranestad and Sturmfels}{Ranestad and
  Sturmfels}{2012}]%
        {Ranestad}
\bibfield{author}{\bibinfo{person}{Kristian Ranestad} {and}
  \bibinfo{person}{Bernd Sturmfels}.} \bibinfo{year}{2012}\natexlab{}.
\newblock \showarticletitle{On the convex hull of a space curve}.
\newblock \bibinfo{journal}{{\em Advances in Geometry\/}}  \bibinfo{volume}{12}
  (\bibinfo{year}{2012}), \bibinfo{pages}{157--178}.
\newblock


\bibitem[\protect\citeauthoryear{Sert\"oz}{Sert\"oz}{2017}]%
        {EmreDiss}
\bibfield{author}{\bibinfo{person}{Emre Sert\"oz}.}
  \bibinfo{year}{2017}\natexlab{}.
\newblock \bibinfo{title}{Enumerative Geometry of Double Spin Curves}.
\newblock \bibinfo{howpublished}{Doctoral Dissertation, HU Berlin,
  \url{https://edoc.hu-berlin.de/handle/18452/19134}}.
  (\bibinfo{year}{2017}).
\newblock


\bibitem[\protect\citeauthoryear{Testa, V\'arilly-Alvarado, and Velasco}{Testa
  et~al\mbox{.}}{2009}]%
        {TVV09}
\bibfield{author}{\bibinfo{person}{Damiano Testa}, \bibinfo{person}{Anthony
  V\'arilly-Alvarado}, {and} \bibinfo{person}{Mauricio Velasco}.}
  \bibinfo{year}{2009}\natexlab{}.
\newblock \showarticletitle{Cox rings of degree one del {P}ezzo surfaces}.
\newblock \bibinfo{journal}{{\em Algebra Number Theory\/}} \bibinfo{volume}{3},
  \bibinfo{number}{7} (\bibinfo{year}{2009}), \bibinfo{pages}{729--761}.
\newblock
\showISSN{1937-0652}


\end{thebibliography}

\end{document}